\documentclass[11pt]{article}
\usepackage{pdfpages}
\usepackage{siunitx}
\usepackage{hyperref}
\usepackage{relsize}
\usepackage{enumitem}
\usepackage{graphicx}
\usepackage{float}
\usepackage{wrapfig}
\usepackage{picinpar}
\usepackage{tcolorbox}
\usepackage{hyperref}
\usepackage{epsfig}
\usepackage{fvextra}
\usepackage{tikz}
 \newenvironment{code}
 {\VerbatimEnvironment \begin{Verbatim}[breaklines, breakafter=*, breaksymbolsep=0.5em, baselinestretch=0.8]}{\end{Verbatim}}

 \newif\ifstartedinmathmode
\newcommand\encircled[1]{%
  \relax\ifmmode\startedinmathmodetrue\else\startedinmathmodefalse\fi%
  \tikz[baseline,anchor=base]{%
  \node[draw,circle,outer sep=0pt,inner sep=.2ex]
    {\ifstartedinmathmode$#1$\else#1\fi};}%
}
\def\n{\noindent}
\def\Mav{M_{av}}
\usepackage{amsmath,amssymb,amsbsy,amsfonts,amsthm,latexsym,
            amsopn,amstext,amsxtra,euscript,amscd,amsthm}

\newtheorem{lem}{Lemma}
\newtheorem{lemma}[lem]{Lemma}

\newtheorem{thm}{Theorem}
\newtheorem{theorem}[thm]{Theorem}

\newtheorem{conj}{Conjecture}

\newtheorem*{theoremA}{Theorem A}
\newtheorem*{lemmaLTE}{Lemma LTE}

\def\\{\cr}
\def\({\left(}
\def\){\right)}
\def\[{\left[}
\def\]{\right]}
\def\<{\langle}
\def\>{\rangle}

\begin{document}

\title{On the average value of the minimal Hamming multiple}

\date{\today}

\pagenumbering{arabic}

\author{Eugen~ Ionascu\footnote{Department of Mathematics, Columbus State University, Columbus, GA, USA
 email: ionascu@columbusstate.edu}\qquad Florian~Luca\footnote{School of Maths, Wits University, Johannesburg, South Africa and Centro de Ciencias Matem\'aticas UNAM, Morelia, Mexico  email: florian.luca@wits.ac.za}   \qquad Thomas~ Merino \footnote{Department of Computer Science, Columbus State University, Columbus, GA, USA email: merino$\_$thomas@students.columbusstate.edu } }

\maketitle

\begin{abstract}
We find a nontrivial upper bound on the average value of the function $M(n)$ which associates to every positive integer $n$ the minimal Hamming weight of a multiple of $n$. Some new results about the equation $M(n)=M(n')$ are given.
\end{abstract}



\section{Introduction}

The number of $1$'s in the binary expansion of the positive integer $n$ is customarily denoted by $s_2(n)$ and sometimes called the Hamming weight of $n$. In what follows, the set of natural numbers $\mathbb N$ is going to be used with the understanding that zero is not included.  For every natural number $n$, we let 
$$
M(n):=\min_{k\ge 1} \{s_2(kn)\}.
$$

This sequence was introduced in OEIS in 2003 and logged as {\href{https://oeis.org/A086342}{A086342}}]   (see \cite{OEIS1}). It has $\num{10000}$ terms there and we calculated $M(n)$ for all $n\le 2^{18}=\num{262144}$ (see \href{https://ejionascu.ro/papers/M(n)upto2to18.txt}{data}).  Table~\ref{table:1} below lists a few values of $M(n)$:

\begin{table}[h!]
\centering
\begin{tabular}{||c |c c c c c c c c c c c c c c c c||} 
 \hline
 $n$ & 1 & 2 & 3 &  \boxed{4} & 5 &  \boxed{6} &  \boxed{7 } & \boxed{8} & 9 & 10 &11& 12 & 13 &  14  &\boxed{15} & \boxed{16}\\ [0.5ex] 
 \hline
 $M(n)$ & 1 & 1 & 2 & 1& 2 & \encircled{2} &  \encircled{3} &  1 & 2 & 2 &2 & 2 & 2 &  3  &   \encircled{4 } & 1\\ 
 \hline \hline
 $n$ & 17 & 18 & 19 & 20 & 21 & 22 &23 & 24 & 25 &26 &27 &28 &29& \boxed{30} & \boxed{31} & \boxed{32} \\
 $M(n)$ & 2 &  2 & 2 &  2 & 3 & 2 & 3&  2 &  2& 2 &  2 & 3 &  2&   \encircled{4} & \encircled{5} & 1 \\
  \hline
\end{tabular}
\caption{Table with the first 32 values of $M(n)$.}
\label{table:1}
\end{table}

It is clear that $M(2^k)=1$, $M(2^k+1)=M(2^k+2)=2$ for every non-negative integer $k\ge 2$.  One can see from this table that, the values of $M(p)$ for non-Mersenne primes $p$ are small ($2$ or $3$).  We used squares and circles in the above table to point out the values of $M(\cdot)$ around the powers of $2$. Writing $2n$ in base $2$ requires a shifting of the writing of $n$ followed by a $0$. This means $M(2n)=M(n)$ and so the values of $M(\cdot)$ are determined by the values taken on odd numbers. 
The average of all 32 values of $M(\cdot)$ in this table is $\frac{70}{32}=\frac{35}{16}= 2.1875$. We looked at the average of all values $M(n)$ in the range  $n\le 2^{18}$ and it is quite surprisingly low: $3.11846$.

We are interested in the limit  (if it exists)
$$\lim_{n\to \infty} \frac{1}{n}\sum_{k=1}^n M(k),$$

\n or at least an upper bound of the average sequence $
{\Mav}(n):=\frac{1}{n}\sum_{k\le n} M(k)
$. 
 
\vspace{0.1in} 
\n By Stolz–Ces\`{a}ro Lemma, if the following limit exists 

\begin{equation}\label{limitforMn}
\lim_{n\to \infty} \frac{1}{2^n}\sum_{k=2^n}^{2^{n+1}-1} M(k)=L,
\end{equation}

\n then the limit of ${\Mav}(2^n)$ must exist and be  $L$. Let us introduce the notation ${C_{av}}(n)$ for such averages, that is, ${C_{av}}(n):=\frac{1}{2^{n-1}}\sum_{k=2^{n-1}}^{2^{n}-1} M(k)$. Numerical evidence shows that ${C_{av}}(n)$ is strictly increasing and it is bounded above by some constant $L$ (the limit in (\ref{limitforMn})) which appears to be $\pi$. For example, for our data $${C_{av}}(18)\approx 3.13184<\pi.$$ The exact values of $C_{av}(n)$ for our data and a graphical depiction  (Figure~\ref{fig:oddvalues}) of these points are included next: 

$$\left\{1,\frac{3}{2},2,\frac{9}{4},\frac{39}{16},\frac{21}{8},\frac{45}{16},\frac{183}{64},\frac{47}{16},\frac{761}{256},\frac{3079}{1024},\frac{3111}{1024},\frac{6253}{2048},\frac{25213}{8192},\frac{25327}{8192},\frac{101849}{32768},\frac{12781}{4096},\frac{410497}{131072}\right\},$$

\vspace{0.2in}

\begin{figure}[h!]
 \epsfig{file=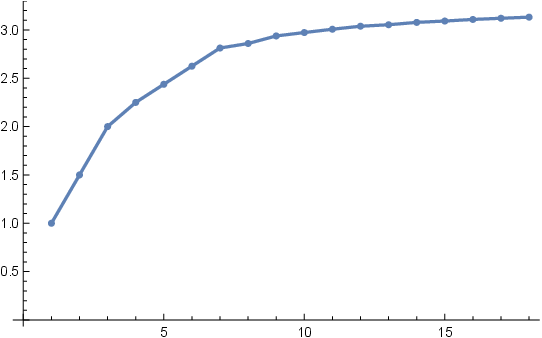,height=3in,width=6in}
  \caption{The sequence $\{C_{av}(n)\}_n$ for $n=1,2,\ldots, 18$. }
  \label{fig:oddvalues}
\end{figure}

\vspace{0.2in}
\n If we consider the average of the values of $M(\cdot) $ only for odd numbers, using the notation 
$${CO_{av}}(n):=\frac{1}{2^{n-1}}\sum_{k\ge 2^{n-1},\ k \ \mathrm{odd}}^{2^{n}-1} M(k),$$

\n then ${C_{av}}(n)=\frac{C_{av}(n-1)}{2}+\frac{CO_{av}(n)}{2}$. This can be turned into a telescopic sum relation $$2^nC_{av}(n)=2^{n-1}C_{av} (n-1)+2^{n-1}CO_{av}(n)$$ that gives

\begin{equation}\label{eq1oda}
C_{av}(n)=\frac{1+CO_{av}(1)+2CO_{av}(2)+\cdots +2^{n-1}CO_{av}(n)}{2^n},
\end{equation}

\n which is a weighted average giving more weight for the latest terms. 
Let us observe that in the right-hand side of (\ref{eq1oda}), at the numerator, we have exactly the sum of all 
values of $M(\cdot)$ for the odd integers $k$ in the range $k\in [1,2^n-1]$. Also, from  (\ref{eq1oda}), again by Stolz–Ces\`{a}ro Lemma, we get that $\lim_{n} C_{av}(n)$ exists and it is the same, if  $\lim_{n} CO_{av}(n)$ exists. We conclude that the problem reduces to estimating the values of $M(\cdot)$ for odd integers. 

\vspace{0.1in}

Very recently (see \cite{UZ}), Problem $B_2$ from the Putnam exam of 2023 asked to calculate $M(2023)$. An example of a multiple of $2023$ with the smallest Hamming weight is $2^{280}+2^5+2^0$. There are infinitely many such numbers  but 
the easiest to get is perhaps $N=2^{277}+2^8+1$ which, as a curiosity, has the  property

$$\frac{N}{2023}=\num{120036}\underset{71 \ \text{digits}}{\underbrace{\dots\dots\dots\dots }}\num{2023}.$$

The fact that $M(2023)$ cannot be $2$ follows from the fact that $7\mid 2023$ and the only powers of $2$ modulo $7$ are $1$, $2$, and $4$. No two of them can sum to $0$ modulo $7$. A more general result has been proved by Luca \cite{Lu}. This was  generalized by Schinzel \cite{Sch}, who showed that $M(2^k-1)=k$ for any $k\ge 1$.  Numbers like these, for which $M(n)=s_2(n)$  are called by Stolarky in~\cite{sto} {\it sturdy}  numbers.  It is shown there that, for instance
$$n=\frac{2^{sk}-1}{2^s-1},\ s,k\in \mathbb N,$$
\n  are sturdy. 

\vspace{0.1in}

Because every multiple of $ab$ is a multiple of $a$, we see that, in fact, more can be said:

$$\max\{M(a),M(b)\}\le M(ab)$$

\n for every $a,b\in \mathbb N$. This relation is pushing the bounds of $\Mav (n)$
up, but we have some opposite results too. First, let's start with a classic:

\vspace{0.1in}

 \begin{theoremA} (Hasse \cite{hasse})  The set of primes  $p$ that divide $2^n + 1$ for some integer $n > 0$  has density $\frac{17}{24}\approx 0.708333$.
\end{theoremA}
\vspace{0.1in}
Rephrasing Hasse's result in our terms, more than $70\%$ of the primes $p$ have $M(p)=2$.
For $k\in \mathbb N$, let us denote the set of primes $p$ such that $M(p)=k$  by ${\cal P}_k$. Then we have ${\cal P}_1=\{2\}$.  The set ${\cal P}_2=\{2,5,\dots\}$ is the set of primes in Theorem A. Viewing these primes as a sequence, ${\cal P}_2$ was entered in OEIS as  {\href{https://oeis.org/A091317}{A091317}} (see \cite{OEIS2}). Also, with a different definition, it is identical to {\href{https://oeis.org/A014662}{A014662}} which was introduced by Sloan in 2008 (\cite{OEIS3}). This definition is ``the primes~$p$ for which the multiplicative order of $2$ modulo $p$ is an even number".  It is easy to see why this is indeed a characterization of ${\cal P}_2$.  Moreover, for all these primes we have $M(p^k)=2$, $k\ge 1$, because of the Lifting of the Exponent Lemma (LTE) (\cite{lte}):

\vspace{0.1in}

\begin{lemmaLTE} [   ] Let $x$ and $y$ be arbitrary
integers, $n$ be a positive integer, and $p$ be an odd prime such that $p|(x-y)$, $p\not |x$, and $p\not|y$.  We have then 
$$v_p(x^n-y^n) = v_p(x-y) + v_p(n).$$
\end{lemmaLTE}

\n Here, as usual, $v_p(m)$ is the power of the prime $p$ in the prime decomposition of $m$. Also, we can arrive at $M(p^k)=2$, $k\ge 1$, by using the binomial formula. Indeed, for some integers $s$ and $n$, $2^n+1=ps$ implies 
$$(2^n)^p+1=(ps-1)^p+1\equiv 0 \ \pmod{p^2},$$
showing that $M(p^2)=2$. We repeat this and use induction over $k$. We prove a similar but more general result in the last section (Theorem~\ref{main2}). As a result of the property of primes in ${\cal P}_2$ being characterized by having an even multiplicative order for $2$ modulo $p$, we can say that each prime $q$ in ${\cal P}_3=\{7,23, 47, 71,\dots\} $ ({\href{https://oeis.org/A308732}{A308732}}, see \cite{OEIS4}) has an odd multiplicative order for $2\pmod q$. 

Even more, results that are in ``favor" of a finite upper bound of  ${\Mav}(x)$ are obtained in (\cite{ce}) where it is shown that $M(p)\le 7$ for almost all primes. But since the upper bound we calculated seems more likely to be around $\pi$, the following conjecture of Ska\l{}ba is even more in support of that.

\par \vspace{0.1in}

\begin{conj}  For almost all primes $p$,  $M(p)\le 3$.
\end{conj}

\par \vspace{0.1in}

\n One more fact that helps in controlling the growth of  ${\Mav}(x)$ is  Lemma~\ref{lem:2} in the next section. Since upper bounds for $M(n)$ are important,  we have the first estimate that will be used later in Section~3 and proved in Section~2:

\begin{equation}\label{estone} M(n)=O(\log n).
\end{equation}

\n  However, this inequality is sharp infinitely often as we saw from the example $$M(2^k-1)=k> \frac{\log(2^k-1)}{\log 2}, \ k\in \mathbb N.$$ On the other hand, we will show that for  ${\Mav}(x)$, the following is true. 

\begin{theorem}\label{main}
We have for some constants $\mathtt C$ and $\mathtt x_0$ 
\begin{equation}\label{mainresult}
{\Mav}(x)=\frac{1}{x}\underset{n\le x}{\sum} M(n)\le  \mathtt C (\log\log x)^3\ \  {\text{\rm as}}\ \  x\ge \mathtt x_0.
\end{equation}
\end{theorem}


\section{Preliminary results and notations}\label{sec2}

In what follows, the symbol $\ll$,  is, at times, replacing the usual notation $f=O(g)$ meaning that $|f(n)|\le Cg(n)$ if $n$ is big enough. A similar statement goes for the symbol $\gg$. If $a$ and $b$ are integers, we say that $b$ divides $a$ if $a=bc$ for some integer $c$. We write this relation as $b|a$ and we call $b$ a {\it divisor} of $a$. A number $p\in \mathbb N\setminus\{1\}$ is said to be a {\it prime} if 
the only divisors of $p$ are $1$ and itself. The set of prime numbers is denoted by $\cal P$. As usual for $n\in \mathbb N$, define $\mu(n)$,  the Möbius function, by 

 $$\mu(n)=
\begin{cases}1\  \text{if n is a square-free positive integer with an even number of prime factors,}\   \\ 
 -1\  \text{ if n is a square-free positive integer with an odd number of prime factors,}         \\ 
 0\  \text{ if n has a squared prime factor.}
\end{cases}
$$
We remind the reader that $\omega(n)$ is the number of distinct prime divisors of $n$.  Also, $\phi(n)$ is the Totient function, the number of natural numbers less than $n$ which are relatively prime with $n$. For a set $A$, we use  $\# A$ as the standard notation for the cardinality of $A$. The Dickman–de Bruijn function $\displaystyle \rho (u)$ is a continuous function that satisfies the delay differential equation

$$\displaystyle u\rho '(u)+\rho (u-1)=0\,$$

\n with initial conditions $\displaystyle \rho (u)=1$ for $0\le u\le 1$. We will be using these ingredients in the proof of the main theorem.

For an odd number $a$, we denote by $\ell_2(a)$ the multiplicative order of $2$ modulo $a$, i.e., the smallest natural number $k$ such that $2^k\equiv 1$ (mod a). Dedekind's Psi function is defined by 
$$\Psi(n)=n\underset{p|n}{\prod}\Big (1+\frac{1}{p}\Big )$$

\n where the index set of the above product goes over all prime divisors of $n$. 
\par\vspace{0.1in} 
We start with the simple argumentation of (\ref{estone}). We write $n$ in base $2$ as $n=\sum_{j=0}^{s-1}2^{u_j}$ where $s=s_2(n)$. All the exponents are distinct, so we may assume that $u_0\ge 0$, $u_1>u_0$, which attracts $u_1\ge 1$. Inductively, we find that $u_j\ge j$ for all $j=0,1,\ldots,s-1$. Hence, we infer that 
$$n\ge \sum_{j=0}^{s-1}2^{j}=2^{s}-1>2^{s-1}.$$
Therefore, we have $M(n)\le s_2(n)=s<1+\frac{\log n}{\log 2}$ which implies  (\ref{estone}).

\par\par\vspace{0.1in} 

We continue with two important lemmas.

\begin{lemma}
\label{lem:1}
We have 
$$
M(ab)\le M(a)M(b)
$$
for all positive integers $a$ and $b$. 
\end{lemma}

\begin{proof}
Indeed, if $M(a)=i$, $M(b)=j$,
$$
a\mid \sum_{u=1}^i 2^{\alpha_u},\ \ \text{and}\ \ b\mid \sum_{v=1}^j 2^{\beta_v}
$$
with mutually distinct tuples $[\alpha_1,\ldots,\alpha_i]$ and $[\beta_1,\ldots,\beta_j]$, then 
$$
ab\mid \sum_{\substack{1\le u\le i\\ 1\le v\le j}} 2^{\alpha_i+\beta_j}:=K.
$$
The last sum above has $ij$ powers of 2 in it, which may not all be distinct. However, combining two terms with the same exponents, like $2^s+2^s=2^{s+1}$, reduces the number of terms in the sum by one, and eventually one arrives at a sum of powers of $2$ with distinct exponents. Then the digit sum of $K$ is at most $i j=M(a)M(b)$. By the definition of $M(\cdot)$, we then have $M(ab)\le s_2(K)\le M(a)M(b)$ and the lemma is proved. \end{proof}

\vspace{0.1 in}

\begin{lemma}
\label{lem:2}
If $a$ and $b$ are odd coprime and $\gcd(\ell_2(a),\ell_2(b))=1$, then 

\begin{equation}\label{new}
M(ab)= \max\{M(a),M(b)\}.
\end{equation}
\end{lemma}

\begin{proof}
Let again $i=M(a)$ and $j=M(b)$. Let 
$$
 a\mid \sum_{u=1}^i 2^{\alpha_u}\ \ \text{and}\ \ b\mid \sum_{v=1}^j 2^{\beta_v},
$$
and without loss of generality, we may assume that $i<j$.  We can always replace one of the powers of $2$, say $2^{\alpha_i}$, by a sum of two distinct powers of $2$ say $2^{\alpha_i-1+x\ell_2(a)}+2^{\alpha_i-1+y\ell_2(a)}$,  where $x$ and $y$ are sufficiently large positive integers. We can also ensure that $$\{\alpha_1,\ldots,\alpha_{i-1},\alpha_i-1+x\ell_2(a),\alpha_i-1+y\ell_2(a)\}$$ are $i+1$ distinct integers and the divisibility 
$$
a\mid \bigg(\sum_{u=1}^{i-1} 2^{\alpha_u}\bigg)+2^{\alpha_i-1+x\ell_2(a)}+2^{\alpha_i-1+y\ell_2(a)}
$$
still holds, but now the number on the right has digit sum $i+1$. Proceeding in this way, if $i<j$ for say $j-i$ steps, we may assume that we have a situation like
$$
a\mid \sum_{u=1}^j 2^{\alpha_u}\ \text{and} \  b \mid \sum_{u=1}^j 2^{\beta_u},
$$
where the two multiples of $a$, respectively $b$, have Hamming weight equal to $j=\max\{M(a),M(b)\}$. It now suffices to 
choose exponents $\{x_u\}_{u=1,2,\cdots,j}$, by the Chinese Remainder Lemma (which can be applied under the given hypothesis on $\ell_2(a)$ and $\ell_2(b)$), satisfying $$x_u\equiv \alpha_u\pmod {\ell_2(a)}\ \text{and}\ x_u\equiv \beta_u\pmod {\ell_2(b)}, u=1,2,\cdots,j .$$  Furthermore, we may pick them so that they are all distinct positive integers. For this choice, we have that both $a$ and $b$ divide $N:=\sum_{u=1}^j 2^{x_u}$,
a number of Hamming weight $j$.  \mbox{Since $a$ and $b$}  are assumed to be coprime, we conclude that $ab\mid N$, which shows that $$M(ab)\le s_2(N)= \max\{M(a),M(b)\}.$$

\n Since the other inequality is true by the definition of $M$, we conclude that the equality (\ref{new}) is true. \end{proof}
For more notation, given two positive integers $k,w\ge 2$, we write $\log_k w=\max\{1,\log (\log_{k-1} w)\}$ with $\log_1 w$ being the natural logarithm.
We will be using this notation later and as a warning, it should not be confused with logarithms in base $k$ or $k-1$. 
\vspace{0.2in}
\section{The proof of the Theorem 1}

This proof is quite long; we are dividing it into pieces for clarity. We split the estimates over a few classes of integers that appear in the sum in (\ref{mainresult}). Namely, we partition the set of $n\le x$ into four subsets. The first three lead to estimates that are much weaker than what we claim in (\ref{mainresult}). The last one of them leads to estimates that match the estimate in (\ref{mainresult}).

We let $x$ be large and consider 
$y:=(\log x)^{K\log_2 x}$, for some large $K$ a large constant to be determined later.  We are assuming that $y<x$ which is equivalent to $K(\log \log x)^2<\log x$ and, since $\lim_{x\to \infty} \frac{\log x}{ (\log \log x)^2}=\infty$, we have indeed $y<x$ if $x$ is big enough no matter what the fixed constant $K$ is.

Next, we let  
${\mathcal Q}:=\{p\in {\mathcal P}: \ell_2(p)\le p^{49/100}\}$ and for $t>1$ we consider  ${\mathcal Q}(t)={\mathcal Q}\cap [1,t)$. Two of the sets we mentioned earlier are  
$$
{\mathcal A}_1(x):=\{n\le x: p\mid n~{\text{\rm for~some~prime}}~p\in {\mathcal Q},~p>y\} \ \text{and}
$$

$${\mathcal A}_2(x):=\{ n\le x: x\not \in {\mathcal A}_1(x), p^2\mid n \text{ for some prime}\  p>y\} .$$

\n If  $n$ is not in ${\mathcal A}_1(x)\cup {\mathcal A}_2(x)$, we can write it as 
 $n=ab$, 
 where $P(a)\le y$ and every prime factor $p$ of $b$ satisfies $p>y$. Here, $P(m)$ is the largest prime factor of the positive integer $m$ with the convention that $P(1)=1$. Since $n\not \in {\mathcal A}_2(x)$, it follows that $b$ is square-free. Also, $p\mid b$ implies that 
 $\ell_2(p)>p^{49/50}$, since $n\not\in {\mathcal A}_1(x)$. We will show that 
 \begin{itemize}
 \item[(i)] $M(b)\le 7$ except possibly for $n$ in a set ${\mathcal A}_3(x)$;
 \item[(ii)] $a\le \log x^{\varepsilon (\log_2 x)^2}$ for any fixed $\varepsilon>0$ except for $n$ in a set ${\mathcal A}_4(x)$,  for $x>x_{\varepsilon}$,
\end{itemize}
\n where both sets ${\mathcal A}_3(x)$ and ${\mathcal A}_4(x)$ are of cardinality $O(x/(\log x)^2)$.

\vspace{0.1in}
\n {\bf Case I:} For the numbers in  ${\mathcal A}_1(x)$, by  a classical argument of Erd\H os and Murty \cite{EM}, we have 
$$
2^{\#{\mathcal Q}(t)}\le \prod_{q\in {\mathcal Q}(t)} q\mid \prod_{k\le t^{49/100}} (2^k-1)\le 2^{\sum_{k\le t^{49/100}}k}=\exp(O(t^{49/50})),
$$
showing that $\#{\mathcal Q}(t)\ll t^{49/50}.$ Each $n\in {\mathcal A}_1(x)$ has a prime factor $q>y$ in ${\mathcal Q}$. Thus, the number of elments  $n\in {\mathcal A}_1(x)$ is at most
 $$
 \#{\mathcal A}_1(x)  \le  \sum_{\substack{q\in {\mathcal Q}\\ q>y}} \frac{x}{q}\ll x\int_{y}^x \frac{d \#{\mathcal Q}(t)}{t}
  \ll  x\int_y^x \frac{\#{\mathcal Q}(t)}{t^2} dt\ll x\int_y^x \frac{dt}{t^{51/50}}
  \ll   \frac{x}{y^{1/50}}\ll \frac{x}{(\log x)^2}.
 $$
 Using  (\ref{estone}), this shows that 
 $$
 \sum_{n\in {\mathcal A}_1(x)} M(n)\ll \sum_{n\in {\mathcal A}_1(x)} \log n\ll (\log x)\#{\mathcal A}_1(x)=O\left(\frac{x}{\log x}\right),
 $$
 and this is negligible for us. 

\vspace{0.1in}
\n {\bf Case II:}  For each  $n\in {\mathcal A}_2(x)$ there is a prime $p>y$ such that $p^2\mid n$, and the number of such $n$ when $p$ is fixed is $\le x/p^2$. Hence,
 $$
 \#{\mathcal A}_2(x)\ll \sum_{p>y}\frac{x}{p^2}\ll x \int_y^{\sqrt{x}} \frac{dt}{t^2}\ll \frac{x}{y}\ll \frac{x}{(\log x)^2},
 $$
 which shows, via (\ref{estone}), that $M(n)\ll \log n\ll \log x$. Hence, we infer that
 $$
 \sum_{n\in {\mathcal A}_2(x)} M(n)\ll (\log x)\#{\mathcal A}_2(x)\ll \frac{x}{\log x}.
 $$

\n {\bf Case III:} Let us consider the numbers outside of  ${\mathcal A}_1(x)$ and ${\mathcal A}_2(x)$ but in ${\mathcal A}_3(x)$ and also in  ${\mathcal A}_4(x)$. 
  Then, using the same idea as before,  we obtain
 $$
 \sum_{n\in {\mathcal A}_i(x)} M(n)\ll (\log x)\#{\mathcal A}_i(x)=O\left(\frac{x}{\log x}\right)\qquad {\text{\rm for}}\quad i=3,4,
 $$
 and these are negligible. 

\n {\bf Case IV:} On the rest of $n\le x$, we have that 
 $$
 M(n)\le M(a) M(b)\ll \log a\ll \varepsilon (\log_2 x)^3,
 $$
 where in the above we used Lemma \ref{lem:1}, which gives us what we wanted in light of the fact that $\varepsilon>0$ can be taken arbitrarily small. 

\par\vspace{0.1in}

It remains to prove that (i) and (ii) hold. 
\par
\n {\bf Case (i):} For (i), we let $n=ab$, where $b$ is square-free, write 
 $
 b=\prod_{p\mid b} p,
 $
 and consider  $d_p$ defined  by 

\begin{equation}\label{dpdef}
d_p(n):=\gcd\Big(\phi(p),\phi(\frac{n}{p})\Big)=\gcd\Big(p-1,\phi(\frac{n}{p})\Big)=\gcd\Big(p-1,\prod_{\substack{r\mid b\\ r\ne p}} (r-1)\Big).
\end{equation}

We aim to put a bound on $d_p$ valid for all $n\le x$ not in ${\mathcal A}_1(x)\cup {\mathcal A}_2(x)$ and which is small for most such $n$. Well, put $z:=(\log x)^5$ 
 and let 
 \begin{itemize}
 \item[(1)] ${\mathcal A}_{3,1}(x)$ is the set of all $n\le x$ such that $q\mid d_p$ for some prime $q>z$;
 \item[(2)] ${\mathcal A}_{3,2}(x)$ is the set of all $n\le x$ such that $q^{\beta} \mid d_p$ for some $\beta >1$ and $q^\beta >z$;
 \item[(3)] ${\mathcal A}_{3,3}(x)$ is the set of all $n\le x$ divisible by at least $N:=\lfloor 10\log_2 x\rfloor$ primes. 
 \end{itemize}
 For ${\mathcal A}_{3,1}(x)$, the numbers in this set have two prime factors $p$ and $r$ such that $q\mid (p-1,r-1)$ for some $q>z$. For fixed $p,r$, the number of such $n\le x$ is $\le x/pr$. Summing up over $p$ and $r$ in the progression $1\pmod q$, we get 
 $$
 x\sum_{\substack{p,q\equiv 1\pmod q\\ p,r\le x}} \frac{1}{pr}\ll x\bigg(\sum_{\substack{p\le x\\ p\equiv 1 \pmod q}} \frac{1}{p}\bigg)^2\ll \frac{x(\log_2 x)^2}{q^2}.
 $$

 \n Summing up the above estimate over all primes $q>z$, we get 
 $$
 \#{\mathcal A}_{3,1}(x)  \ll  \sum_{q>z} \frac{x(\log_2 x)^2}{q^2}\ll x(\log_2 x)^2\sum_{q>z} \frac{1}{q^2}\ll \frac{x(\log_2 x)^2}{z}
  = O\left(\frac{x}{(\log x)^2}\right)
 $$
 and this is acceptable for us. 

\par\vspace{0.1in}
 When $n\in {\mathcal A}_{3,2}(x)$, then $p\mid n$ for some prime $p$ such that $p\equiv 1\pmod {q^\beta}$ with $\beta>1,~q^\beta>z$. Fixing $q^\beta$, the number of such $n\le x$ is at most
 $$
 \sum_{\substack{p\le x\\ p\equiv 1\pmod {q^a}}} \frac{x}{p}\ll \frac{x\log_2 x}{\phi(q^\beta)}\ll \frac{x\log_2 x}{q^\beta }.
 $$
 Summing this up over all $q^\beta$ with $\beta>1$ (in particular, squarefull numbers) and $q^\beta>z$, we get that
 $$
 \#{\mathcal A}_{3,2}(x)  \ll  \sum_{\substack{q^\beta>z\\ \beta>1}} \frac{x(\log_2 x)}{q^\beta}\ll x(\log_2 x) \sum_{\substack{s>z\\ s~{\text{\rm squarefull}}}} \frac{1}{s}
  \ll  \frac{x\log_2 x}{{\sqrt{z}}}=O\left(\frac{x}{(\log x)^2}\right).$$

\par\vspace{0.1in}

 For ${\mathcal A}_{3,3}(x)$, recall that $N=\lfloor 10\log_2 x\rfloor$ and let $n\le x$ be a positive integer divisible by a prime $p\equiv 1\pmod m$, where $m$ is square-free and has at least $N$ prime factors. For fixed $p$, the number of such $n\le x$ is $\le x/p$. Summing this up over all $p\equiv 1\pmod m$, $p\le x$, we get
 $$
 x\sum_{\substack{p\le x\\ p\equiv 1\pmod m}} \frac{1}{p}\ll \frac{x\log_2 x}{\phi(m)}.
 $$
 Summing this up over all square-free $m$'s having at least $N$ prime factors we get

$$
 \sum_{\substack{m\le x\\ \mu^2(m)=1\\ \omega(m)\ge N}} \frac{x\log_2 x}{\phi(m)}  \ll   x(\log_2 x)\sum_{k\ge N} \frac{1}{k!}\left(\sum_{p\le x} \frac{1}{p-1}\right)^k. $$
As a result, we have 
\begin{equation} \label{firstest} 
  x\sum_{\substack{p\le x\\ p\equiv 1\pmod m}} \frac{1}{p}  \ll  x\log_2 x \sum_{k\ge N} \frac{1}{k!}\left(\log_2 x+c_0\right)^k,
 \end{equation}

 \n where $c_0$ is some absolute constant. 
 Since $k!\ge (k/e)^k$, we arrive at a bound of 
 $$
 x\log_2 x\sum_{k\ge N }\left(\frac{e\log_2 x+ec_0}{k}\right)^k
 $$
 \n on the  sum in (\ref{firstest}). 
 For $k>N$, the above sum is dominated by its first term. Thus, 
 $$
 \#{\mathcal A}_{3,3}(x)\ll x\log_2 x \left(\frac{e\log_2 x+ec_0}{N}\right)^N\ll \frac{x\log_2 x}{(\log x)^{10\log(N/e)}}=O\left(\frac{x}{(\log x)^{10}}\right),
 $$
 which is acceptable for us. Put ${\mathcal A}_3(x):=\cup_{i=1}^3 {\mathcal A}_{3,i}(x)$ and let $n\not \in {\mathcal A}_j(x)$ for $j=1,2,3$. Then writing for $p\mid b$,
 
\begin{equation}\label{fpdef}
 f_p=\prod_{q^\beta\| b} q^\beta,
 \end{equation}

 \n we get that $q^\beta<z$ (for both $\beta=1$ and $\beta>1$) and there are at most $N$ prime power factors \mbox{$q^\beta$ in $f_p$.} Thus, 
 $$
 f_p<z^N=(\log x)^{50\log_2 x}<p^{1/100},
 $$


 \n provided we take $K:=\num{5000}$. Since $\ell_2(p)\ge p^{49/100}$, we get that 
 $$
 \frac{\ell_2(p)}{\gcd(\ell_2(p),f_p)}\ge p^{49/100-1/100}=p^{48/100}=p^{12/25}>p^{11/23}\log p
 $$

\n for large $p$. Let $g_p:=2^{f_p}$. It then follows that the element $g_p$ has order at least \mbox{$p^{11/23}\log p$ modulo $p$,} so by a result of Hart \cite{Hart}, we have that 
 \begin{equation}
 \label{eq:Hart}
 p\mid \sum_{i=1}^7 g_p^{\alpha_i}
 \end{equation}
 
\n for some positive integers $\alpha_1,\ldots,\alpha_7$. Indeed, Theorem 1.13 in \cite{Hart} says that if $A\subset ({\mathbb Z}/p{\mathbb Z})^*$ is a multiplicative subgroup of $({\mathbb Z}/p{\mathbb Z})^*$ of order 
$\gg p^{11/123}(\log |A|)^{12/23}$, then 
$$
({\mathbb Z}/p{\mathbb Z})^*\subseteq 6A,
$$
where the right--hand side above means all elements of the form  $\sum_{i=1}^6 a_i\pmod p$, where $a_i\in A$. Applying this to the subgroup $A$ of $({\mathbb Z}/p{\mathbb Z})^*$ generated by $g_p$, we get that $-1\in 6A\pmod p$, so $0\in 1+6A\pmod p$ which implies \eqref{eq:Hart}.
This holds for all $p\mid b$. Further, 
 $$
 \gcd(\ell_2(p)/d_p,\ell_2(q)/d_q)=1\qquad {\text{\rm for~any~two~prime~factors}}~p,q~{\text{\rm of}}~~b.
 $$
 By the argument of Lemma \ref{lem:2} based on the Chinese Remainder Theorem, we get that $M(b)\le 7$. 

\vspace{0.1in}

{\bf Case (ii)}  For this, we use estimates from the theory of smooth numbers. Namely, let $w:=(\log x)^{\varepsilon (\log_2 x)^2}$. Assume that 
 $a>w$. Then $n\le x$ is a positive integer divisible by $a>w$ such that $P(a)\le y$. The number of such $n\le x$ is $n/a$ and summing over all $y$-smooth $a$, we get that 
 $$
 \#{\mathcal A}_4(x)\ll x\sum_{\substack{w<a\le x\\ P(a)\le y}} \frac{1}{a}.
 $$
 The right--hand side above is 
 $$
 \ll x(\log x) \rho(u),
 $$
 where $u:=\log w/\log y$.  Indeed, keeping $u$ fixed, for $t\in [w,x]$, the number of $y$-smooth positive integers $a\le t$ is bounded above by $\Psi(t,t^{1/u})$. Here, $\Psi(x,y)$ counts the number of positive integers $n\le x$ with $P(n)\le y$. It follows from work of de Bruijn (see \cite{2}, Part I) that 
$$
\Psi(t,t^{1/u})\ll t\rho(u),
$$
provided $u\le (\log t)^{3/8-\varepsilon}$. For us, $u=(\varepsilon/\num{5000})\log_2 x$, while $\log w=\varepsilon (\log_2 x)^3$, so certainly $u\le (\log t)^{3/8-\varepsilon}$ holds with any $\varepsilon\in (0,1/10)$ 
and any $t\ge w$ provided $x$ is sufficiently large. The above estimate involving the sum of the reciprocals of such $a$ then follows by Abel's summation. Furthermore, since $\rho(u)=u^{(1+o(1))u}$ as $u\to\infty$ (for this estimate, see de Bruijn \cite{1}), we get that $\rho(u)=(\log x)^{-(1+o(1))\log_3 x}$ as $x\to\infty$.

In particular, for large $x$ we get that $\rho(u)<(\log x)^{-3}$, so 
 $$
 \#{\mathcal A}_4(x)\ll \frac{x}{(\log x)^2},
 $$
 as we wanted. These complete the proof of Theorem~\ref{main}.  $\square$
 

 \section{Comments and some other related results}

 In this subsection, we give an idea of how to go about showing that $\Mav(\cdot)$ is actually bounded.

\subsection{Improving the arguments}

 Massaging the arguments used in the proof of Theorem~\ref{main} gives a bit more. Namely, let again $n=a(n)b(n)$, where $P(a(n))\le y$. The above argument shows that 
 $$
 \sum_{n\le x} M(n)=\sum_{n\in {\mathcal S}} 7M(a(n))+O\left(\frac{x}{\log x}\right),
 $$
 where ${\mathcal S}$ is a subset of $[1,x]$ on which $a(n)\le \exp((\log_3 x)^2)=z$. We ignore the parameter $\varepsilon>0$ for now.  This is more of a heuristic analysis. We fix $a\le z$, a $y$-smooth integer, and count the number of $n\le x$ such that $a=a(n)$. These  $n$ have the property that $n=ab\le x$, so $b\le x/a$ and the smallest prime factor of $b$ is $>y$. The number of such $b$ is, by the sieve, $\ll x/a\log y$. Thus, 
 $$
 \frac{1}{x} \sum_{n\le x} M(n)\ll \sum_{\substack{a\le z\\ P(a)\le y}} \frac{M(a)}{a\log y}+O\left(\frac{1}{\log x}\right).
 $$
 Using Abel's summation formula, we get
 \begin{eqnarray*}
 {\Mav}(x) & \ll & \sum_{\substack{a\le z\\ P(a)\le y}} \frac{M(a)}{a\log y}+O\left(\frac{1}{\log x}\right)\\
 &  \ll & \frac{\sum_{t\le z} M(t)}{z\log y}+\int_2^z \frac{(\sum_{s\le t} M(s))d\Psi(t,y)}{t^2}+\frac{1}{\log x}\\ 
 & \ll & \max_{t\le z} {\Mav}(t)\left(\frac{1}{\log y}+\int_2^z \frac{d\Psi(t,y)}{t\log y}+O\left(\frac{1}{\log x}\right)\right).
 \end{eqnarray*}
 The integral inside has the same order of magnitudes as 
 $$
 \sum_{\substack{n\le z\\ P(n)\le y}} \frac{1}{n\log y}\le \frac{1}{\log y}  \prod_{p\le y} \left(1-\frac{1}{p}\right)^{-1}\ll 1.
 $$
 Thus, letting 
 $$
 M'_{av}(t):=\max_{2\le s\le t} \{M_{av}(s)\},
 $$
 the above argument gives 
 $$
 M'_{av(x)}\le C M'_{av}(z).
 $$
 Here, $C>0$ is some universal constant. The above argument assumes $x>x_0$ for some absolute constant $x_0$. We now look at the sequence $x_1:=x,~x_2:=z=\exp((\log_2 x_1)^3)$, and in general, \mbox{given $x_k$}, we take $x_{k+1}:=\exp((\log_2 x_k)^3)$ as long as $x_k>x_0$. Letting $L$ be the number of terms of the above sequence, or the number of iterates of the function $x\mapsto \exp((\log_2 x)^3)$ to get to $x_0$, \mbox{we then get that}
 $$
 M'_{av}(x)\ll C^L.
 $$
 Since $L$ grows more slowly than any fixed iterate of $\log x$, it follows that 
 $$
 M'_{av}(x)\ll \log_ \ell x
 $$
 for any positive integer $\ell$ provided $x>x_{\ell}$. So, maybe it is true that ${\Mav}(x)$ (therefore also $M'_{av}(x)$) is bounded. We leave this as a task for the reader. However, we think we have enough evidence to claim it here. 

\vspace{0.1in} 
\begin{conj} The sequence $\{CO_{av}\}_n$ is monotone and bounded. Its limit is the same as for $\{\Mav(n)\}_n$.
\end{conj} 

\subsection{Powers of primes}
 
 In this section we prove a similar result that relates to the one in Lemma 2. In this case, we are not assuming that  $\gcd(a,b)=1$, but $a=b$ is a prime. 
 
\begin{theorem}\label{main2} For every $k\in \mathbb N$,  $M(p)=M(p^k)$  holds when $p$ is not a Wieferich prime. 
\end{theorem}

\proof First we observe tha $M(2^k)=1$, and so we may assume that $p$ is odd. Since every multiple of $p^k$ is also a multiple of $p$, we get that 

\begin{equation}\label{ineq1}
M(p)\le M(p^k)
\end{equation}

\n for any $k\ge 1$. We now show the converse assuming $p$ is not Wieferich. Let $M(p)=\ell\ge 2$ and 
$$
pm=2^{\alpha_1}+\cdots+2^{\alpha_{\ell}}
$$
be a multiple of $p$ of Hamming weight $\ell$. Because $p$ is not a Wieferich prime, \mbox{there exists  $\lambda\not\equiv 0\ \pmod p$}, such that

\begin{equation}\label{eqW}
2^{p-1}=1+\lambda p.
\end{equation} 
Let $x_1,\ldots,x_{\ell}$ be any nonnegative integers. Then 
\begin{eqnarray*}
&& 2^{\alpha_1+x_1(p-1)}+2^{\alpha_2+x_2(p-1)}+\cdots+2^{\alpha_{\ell}+x_{\ell}(p-1)}\\
& = & 2^{\alpha_1} (2^{p-1})^{x_1}+\cdots+2^{\alpha_{\ell}}(2^{p-1})^{x_{\ell}}\\
&= & 2^{\alpha_1}(1+\lambda p)^{x_1}+2^{\alpha_2} (1+\lambda p)^{x_2}+\cdots+2^{\alpha_{\ell}}(1+\lambda p)^{x_{\ell}} \\
& \equiv & (2^{\alpha_1}+\cdots+2^{\alpha_{\ell}})+\lambda p(x_12^{\alpha_1}+x_2 2^{\alpha_2}+\cdots+x_{\ell} 2^{\alpha_{\ell}}) \pmod {p^2}\\
& \equiv &p\left[m+\lambda (x_1 2^{\alpha_1}+x_2 2^{\alpha_2}+\cdots+x_{\ell} 2^{\alpha_{\ell}})\right ]\pmod {p^2}.
\end{eqnarray*}
Since $p\nmid \lambda$, it is easy to see that we can choose $x_1,\ldots,x_{\ell}$ modulo $p$ such that 
$$
x_1 2^{\alpha_1}+x_22^{\alpha_2}+\cdots+x_{\ell} 2^{\alpha_{\ell}}\equiv -m\lambda^{-1}  \pmod p
$$
(for example, we can fix $x_1,\ldots,x_{\ell-1}$ arbitrarily modulo $p$ then fix $x_{\ell}$ defined by the above congruence). Then by making the $x_i$ arbitrarily large without changing their congruence classes \mbox{modulo $p$,} we may assume that the numbers $\alpha_i+(p-1)x_i$ are distinct for $i=1,\ldots,\ell$. This shows that $M(p^2)\le M(p)$. Hence, $M(p)=M(p^2)$. The general case follows easily by induction. Let us point out that in the inductive step, we use instead of (\ref{eqW}), the identity

\begin{equation}\label{eqW1}
2^{p^{k-1}(p-1)}=1+\lambda' p^{k},\  k\ge 1.
\end{equation} 
\n This is simply a consequence of  (\ref{eqW}) and the binomial theorem. The argument in the induction step goes the same way having (\ref{eqW1}). $\square$

\vspace{0.3in}

\n {\bf Remarks:} In general, let $e_p$ be such that $p^{e_p}\| 2^{p-1}-1$. Then the argument used in the above proof shows that $M(p^k)=M(p^{e_p})$ for all $k\ge e_p$. In particular, for $p=1093$, we have that $p^2\mid 2^{(p-1)/2}+1$, so $M(p)=M(p^2)$. Since 
$p^2\| 2^{p-1}-1$ for this $p$, we have that $M(p^k)=M(p)$ for all $k\ge 1$. \mbox{For $p=3511$,} one can check that $p$ does not divide $2^{(p-1)/2}+1$ and the choice $1+2^{585}+2^{1170}$  is a multiple of $p^2$. As a result, we have  $3=M(p)=M(p^2)$ and since $p^2\| 2^{p-1}-1$, it follows that $M(p^k)=M(p)$ for all $k\ge 1$.

\section{About generating the data}

Our algorithm for $M(n)$, is based first on the following two observations:

\begin{enumerate} 
\item If $n$ is even, then $\frac{n}{2}$ in binary is the same as $n$ with a trailing $0$. As we had observed in the introduction $M(n) = M(n/2)$. This can be applied multiple times until $n$ is odd.

\item If $n = 2^p – 1$, we also observed in the introduction that $M(n) = s_2(n)=p$.

\end{enumerate}

In finding $M(n)$, our goal is to determine the minimum number of terms of the form $2^{\alpha_i}$ such that $\alpha_i$ are non-negative distinct integers and whose sum is the odd number $kn$ for some positive integer $k$. Se we may assume that 
$$2^{\alpha_1} + 2^{\alpha_2} + 2^{\alpha_3} + \ldots + 1 \equiv 0 \pmod{n}\ \ \text{or}$$
 
$$ 2^{\alpha_1} + 2^{\alpha_2} + \ldots + 2^{\alpha_{(r-1)}} \equiv –1 \pmod{n} \equiv n – 1 \pmod{n}.$$
We will check this since it reduces the number of terms to check by $1$; we must add $1$ to the result as a consequence.   

To determine the minimum number of terms needed, we generate the multiplicative cyclic group determined by the powers of $2$ modulo $n$.  We incrementally try all possible summations of one term, then two terms, then three terms, and so on until one summation equals $n – 1 \pmod{n}$. For faster computation, all values with $r$-many terms will be used to generate all values with $(r+1)$-many terms by summing each with the initially generated terms from the cyclic group. 

To avoid unnecessary computations, we are only interested in sums that are  not generated in previous iterations. We also manage different sets of sums to separate old values from newly discovered values. When we search for $n–1$, we will only look at the newly discovered values. Also, we perform a check for values of $n$ of the form $n = 2^p – 1$ before any computation; these have a known value, $M(2^p – 1) = p$, and would be expensive to compute otherwise. A final step taken for the sake of performance is if $n – 1$ is found during sum computation, which takes place in the function called ``$g$" below, then further searching is halted and the result is returned. The following is our algorithm. This version does not provide a value for $k$ for which $s_2(kn) = M(n)$ and may be optimized further with caching, for example.
\par \vspace{0.3in}

  {\bf Algorithm for $M(n)$}:
\begin{itemize}
\item If n is even, calculate $M(\frac{n}{2})$ as it yields the same result.
\item If $n = 1$, return $1$.
\item If $n + 1$ is a power of $2$, say $n = 2^p - 1$, return $p$.
\item Set $r := 2$, which will be the working result.
\item Set $A$ to be the subgroup generated by $2$ modulo $n$  (this will remain unchanged throughout the algorithm).
\item If $A$ contains $n – 1$, return $r$, which will be $2$.
\item Set $C := A$, which will contain all values reachable by adding $(r – 1)$-many terms from $A$ together.
\item Set $B := A$, which will contain all newly discovered values reachable by adding $(r – 1)$-many terms from $A$ together (those not already in $C$).
\item Loop forever (until the loop is told to break): 
\begin{itemize}
\item [$\circ$] Increment $r$.
\item[$\circ$] Get {\it``didFindTarget"} and $D$ from function $g(B, A, n)$, which is defined as returning $D:=A+B$ unless $n-1$ is found, in which case the computation is halted and the variable  {\it``didFindTarget"} is set to True.
\item [$\circ$]  If  {\it``didFindTarget"} is true while executing $g$, break the loop.
\item  [$\circ$]  Set $B:=D\setminus C$, as now $B$ contains only newly found values.
\item  [$\circ$]  Set $C:=C \cup D$.
\end{itemize}
\item Return $r$.
\end{itemize}
This algorithm is now presented below as Python code:
 \begin{code}
import math
# Helper function:
def f(n):
r = {1}
w = 2
while w != 1:
r.add(w)
w = (w * 2) % n
return r

# Helper function:
def g(B, A, n):
for b in B:
if (n - 1 - b) in A: return True, None
R = B.copy()
for b in B:
for a in A:
w = (a + b) % n
if w == n - 1: return True, None
R.add(w)
return False, R

def M(n):
if n == 0: return 0
while n % 2 == 0: n = n // 2
if n == 1: return 1
if math.log2(n + 1).is_integer(): return int(math.log2(n + 1))
r = 2
A = f(n)
if n - 1 in A: return r
B = A
C = A
while True:
found, D = g(B, A, n)
r += 1
if found: break
B = D - C
C = C.union(D)
return r
\end{code}

\par \vspace{0.3in}

We also used Mathematica to generate part of the data (to match the outputs) and the graphs. The idea of the program is basically the same as the one above (except we do not exclude the values of $n$ of the form $2^p-1$ and do not separate the old from the new values generated by having more terms in the summations). The difference between the two is quite important in performance since with this code the time to get all values of $M(n)$ less than $\num{20000}$ is quite big. We include the Mathematica code next.  

\begin{code}
P2mN[n_] := Module[{i, m, A, q}, A = {};
  m = EulerPhi[n];
  Do[A = Union[A, {Mod[2^i, n]}], {i, 0, m}];
  q = MemberQ[A, n - 1]; {A, q}]
\end{code}
\begin{code}
AddTwoSets[A_, B_, n_] := 
 Module[{E, na, nb, i, j, q}, E = {}; na = Length[A]; nb = Length[B];
  Do[E = Union[E, {Mod[A[[i]] + B[[j]], n]}], {i, 1, na}, {j, 1, 
    nb}];
  q = MemberQ[E, n - 1]; {E, q}]
\end{code}
\begin{code}
StripOdd[n_] := Module[{m, x, o}, m = n; x = Mod[m, 2]; o = 0;
  If[x == 1, o = 1]; 
  While[o == 0, m = m/2; x = Mod[m, 2]; If[x == 1, o = 1];]; m]
\end{code}
\begin{code}
MF[n_] := Module[{m, i, a, A, AA, B, X, O}, m = StripOdd[n];
  If[m == 1, O = 1];
  If[m > 1, AA = P2mN[m]; A = AA[[1]]; B = A; a = AA[[2]]; i = 2; 
   O = 2;
   While[a == False, X = AddTwoSets[A, B, m]; B = X[[1]]; a = X[[2]]; 
    i = i + 1; O = i;   ];]; O]
\end{code}
 The first 1000 values then can be found in a couple of minutes. 
\begin{code}
Table[MF[n], {n, 1, 1000}]
  \end{code}
$$\left[ \begin{array}{ccccccccccccccccccccccccc}
 1 & 1 & 2 & 1 & 2 & 2 & 3 & 1 & 2 & 2 & 2 & 2 & 2 & 3 & 4 & 1 & 2 & 2 & 2 & 2 & 3 & 2 & 3 & 2 & 2 \\
 2 & 2 & 3 & 2 & 4 & 5 & 1 & 2 & 2 & 3 & 2 & 2 & 2 & 3 & 2 & 2 & 3 & 2 & 2 & 4 & 3 & 3 & 2 & 3 & 2 \\
 4 & 2 & 2 & 2 & 3 & 3 & 2 & 2 & 2 & 4 & 2 & 5 & 6 & 1 & 2 & 2 & 2 & 2 & 3 & 3 & 3 & 2 & 3 & 2 & 4 \\
 2 & 3 & 3 & 3 & 2 & 2 & 2 & 2 & 3 & 4 & 2 & 3 & 2 & 4 & 4 & 3 & 3 & 5 & 3 & 3 & 2 & 2 & 3 & 2 & 2 \\
 2 & 4 & 3 & 2 & 4 & 2 & 2 & 2 & 2 & 3 & 3 & 3 & 2 & 2 & 3 & 2 & 4 & 2 & 3 & 4 & 2 & 2 & 4 & 5 & 2 \\
 6 & 7 & 1 & 2 & 2 & 2 & 2 & 3 & 2 & 4 & 2 & 2 & 3 & 2 & 3 & 3 & 3 & 3 & 2 & 2 & 3 & 3 & 2 & 2 & 4 \\
 3 & 2 & 4 & 3 & 5 & 3 & 2 & 3 & 3 & 2 & 3 & 2 & 2 & 2 & 4 & 2 & 3 & 3 & 2 & 4 & 2 & 2 & 2 & 3 & 3 \\
 2 & 2 & 4 & 2 & 4 & 2 & 3 & 3 & 3 & 2 & 5 & 3 & 3 & 6 & 3 & 3 & 2 & 2 & 2 & 4 & 3 & 2 & 2 & 3 & 2 \\
 2 & 2 & 3 & 4 & 2 & 3 & 3 & 2 & 2 & 4 & 2 & 2 & 3 & 2 & 4 & 2 & 5 & 2 & 3 & 3 & 4 & 3 & 4 & 3 & 4 \\
 2 & 2 & 2 & 2 & 3 & 3 & 2 & 4 & 4 & 3 & 2 & 3 & 3 & 3 & 4 & 2 & 2 & 2 & 2 & 3 & 4 & 3 & 5 & 2 & 2 \\
 2 & 6 & 3 & 7 & 8 & 1 & 2 & 2 & 3 & 2 & 3 & 2 & 3 & 2 & 2 & 3 & 4 & 2 & 2 & 4 & 3 & 2 & 3 & 2 & 4 \\
 3 & 2 & 2 & 5 & 3 & 2 & 3 & 2 & 3 & 4 & 3 & 3 & 2 & 2 & 2 & 3 & 3 & 2 & 3 & 3 & 2 & 2 & 2 & 3 & 4 \\
 3 & 3 & 3 & 2 & 2 & 4 & 2 & 3 & 3 & 5 & 3 & 3 & 2 & 2 & 6 & 3 & 2 & 3 & 3 & 2 & 2 & 3 & 3 & 2 & 2 \\
 2 & 3 & 2 & 3 & 4 & 2 & 2 & 3 & 3 & 3 & 3 & 3 & 2 & 4 & 4 & 5 & 2 & 3 & 2 & 4 & 2 & 2 & 3 & 2 & 3 \\
 4 & 2 & 2 & 2 & 3 & 4 & 4 & 2 & 3 & 4 & 2 & 2 & 2 & 3 & 3 & 3 & 3 & 3 & 4 & 2 & 3 & 5 & 2 & 3 & 4 \\
 3 & 2 & 6 & 2 & 3 & 7 & 3 & 3 & 2 & 3 & 2 & 2 & 2 & 2 & 4 & 3 & 3 & 2 & 2 & 3 & 2 & 2 & 3 & 5 & 2 \\
 2 & 2 & 5 & 2 & 4 & 3 & 3 & 4 & 2 & 2 & 3 & 3 & 3 & 3 & 3 & 2 & 2 & 2 & 2 & 4 & 2 & 2 & 3 & 2 & 4 \\
 3 & 3 & 2 & 3 & 4 & 4 & 2 & 2 & 5 & 4 & 2 & 3 & 3 & 3 & 3 & 6 & 4 & 2 & 3 & 4 & 4 & 3 & 3 & 2 & 4 \\
 4 & 2 & 3 & 2 & 6 & 2 & 2 & 2 & 4 & 3 & 2 & 3 & 3 & 2 & 5 & 4 & 2 & 4 & 3 & 3 & 4 & 2 & 2 & 3 & 3 \\
 3 & 3 & 3 & 3 & 4 & 2 & 2 & 3 & 2 & 4 & 2 & 3 & 2 & 2 & 3 & 2 & 4 & 3 & 3 & 4 & 5 & 3 & 2 & 2 & 2 \\
 3 & 2 & 3 & 6 & 2 & 3 & 3 & 7 & 2 & 8 & 9 & 1 & 2 & 2 & 3 & 2 & 3 & 3 & 3 & 2 & 2 & 3 & 2 & 2 & 4 \\
\end{array}
\right]$$

$$\left[
\begin{array}{ccccccccccccccccccccccccc}

 3 & 5 & 2 & 3 & 2 & 2 & 3 & 2 & 4 & 3 & 2 & 2 & 2 & 3 & 4 & 2 & 3 & 3 & 2 & 2 & 3 & 2 & 2 & 4 & 4 \\
 3 & 3 & 3 & 2 & 4 & 2 & 2 & 5 & 3 & 3 & 4 & 2 & 2 & 3 & 2 & 2 & 6 & 3 & 2 & 4 & 2 & 3 & 3 & 3 & 3 \\
 2 & 2 & 2 & 3 & 2 & 3 & 3 & 3 & 3 & 4 & 2 & 2 & 3 & 5 & 3 & 3 & 2 & 2 & 2 & 4 & 2 & 3 & 3 & 3 & 4 \\
 4 & 3 & 2 & 3 & 3 & 3 & 3 & 2 & 3 & 2 & 3 & 4 & 2 & 2 & 4 & 3 & 2 & 3 & 2 & 5 & 3 & 3 & 4 & 3 & 2 \\
 2 & 2 & 2 & 3 & 6 & 3 & 3 & 2 & 2 & 7 & 3 & 3 & 3 & 3 & 2 & 2 & 2 & 2 & 3 & 4 & 3 & 3 & 2 & 2 & 2 \\
 5 & 2 & 2 & 3 & 3 & 2 & 4 & 3 & 2 & 4 & 2 & 2 & 4 & 2 & 3 & 3 & 3 & 3 & 4 & 3 & 4 & 3 & 2 & 3 & 4 \\
 2 & 2 & 4 & 3 & 4 & 2 & 5 & 2 & 2 & 2 & 3 & 3 & 2 & 2 & 4 & 2 & 2 & 6 & 2 & 3 & 3 & 4 & 2 & 4 & 3 \\
 2 & 4 & 3 & 2 & 4 & 2 & 3 & 2 & 2 & 3 & 3 & 4 & 5 & 4 & 3 & 2 & 3 & 3 & 3 & 4 & 3 & 2 & 3 & 2 & 2 \\
 2 & 3 & 3 & 2 & 3 & 4 & 3 & 2 & 3 & 4 & 3 & 2 & 4 & 2 & 2 & 4 & 3 & 3 & 5 & 2 & 2 & 2 & 3 & 3 & 4 \\
 3 & 3 & 2 & 2 & 3 & 6 & 2 & 2 & 3 & 3 & 2 & 7 & 4 & 3 & 8 & 3 & 3 & 2 & 2 & 3 & 4 & 2 & 2 & 2 & 5 \\
 2 & 4 & 2 & 3 & 4 & 3 & 3 & 3 & 3 & 2 & 2 & 2 & 2 & 3 & 3 & 3 & 2 & 2 & 2 & 4 & 3 & 2 & 5 & 3 & 2 \\
 4 & 2 & 3 & 2 & 3 & 5 & 3 & 2 & 2 & 4 & 2 & 3 & 3 & 3 & 3 & 4 & 2 & 2 & 6 & 2 & 2 & 3 & 3 & 3 & 4 \\
 3 & 2 & 3 & 2 & 3 & 4 & 2 & 3 & 2 & 3 & 2 & 5 & 2 & 3 & 4 & 2 & 2 & 2 & 2 & 2 & 3 & 3 & 2 & 2 & 4 \\
 3 & 3 & 2 & 3 & 4 & 2 & 2 & 3 & 2 & 4 & 4 & 4 & 3 & 2 & 2 & 2 & 4 & 5 & 3 & 4 & 3 & 2 & 4 & 3 & 3 \\
 3 & 2 & 3 & 3 & 3 & 4 & 6 & 2 & 4 & 4 & 2 & 3 & 3 & 7 & 4 & 2 & 4 & 3 & 3 & 3 & 3 & 3 & 2 & 5 & 4 \\
 3 & 4 & 3 & 2 & 2 & 3 & 2 & 2 & 3 & 6 & 4 & 2 & 2 & 2 & 4 & 2 & 3 & 4 & 3 & 3 & 2 & 2 & 3 & 3 & 2 \\
 3 & 3 & 2 & 2 & 5 & 3 & 4 & 3 & 2 & 4 & 4 & 3 & 3 & 3 & 3 & 2 & 4 & 3 & 2 & 6 & 2 & 2 & 3 & 4 & 3 \\
 3 & 3 & 2 & 3 & 3 & 3 & 3 & 3 & 3 & 4 & 5 & 2 & 2 & 2 & 3 & 3 & 3 & 2 & 4 & 4 & 2 & 2 & 3 & 3 & 4 \\
 2 & 2 & 2 & 4 & 3 & 4 & 2 & 3 & 4 & 2 & 3 & 3 & 3 & 3 & 4 & 3 & 5 & 2 & 3 & 3 & 2 & 2 & 2 & 4 & 2 \\
\end{array}
\right]$$

\end{document}